\documentclass[conference]{IEEEtran}
\usepackage{cite}
\usepackage{amsmath,amssymb,amsfonts}
\usepackage{algorithmic}
\usepackage{graphicx}
\usepackage{textcomp}
\usepackage{xcolor}
\def\BibTeX{{\rm B\kern-.05em{\sc i\kern-.025em b}\kern-.08em
    T\kern-.1667em\lower.7ex\hbox{E}\kern-.125emX}}

\renewcommand{\baselinestretch}{0.965}

\usepackage[utf8]{inputenc}
\usepackage[margin=1in]{geometry}
\usepackage{xcolor}
\usepackage{amsmath}
\usepackage{amsfonts,amssymb}
\usepackage{amsthm}
\usepackage[hidelinks]{hyperref}
\usepackage[nameinlink]{cleveref}
\usepackage{fancyhdr}
\usepackage{bm}
\usepackage{mathtools}
\newtheorem{theorem}{Theorem}
\newtheorem{lemma}{Lemma}
\newtheorem{corollary}{Corollary}
\newtheorem{assumption}{Assumption}
\newtheorem{remark}{Remark}

\newcommand{\inv}{^{-1}}

\newcommand{\nx}{n}

\allowdisplaybreaks

\begin{document}

\title{Efficient Quantum Algorithms for Nonlinear Stochastic Dynamical Systems}

\author{\IEEEauthorblockN{Abeynaya Gnanasekaran}
\IEEEauthorblockA{\textit{Raytheon Technologies Research Center} \\
Berkeley, California, USA \\
abeynaya.gnanasekaran@rtx.com}
\and
\IEEEauthorblockN{Amit Surana}
\IEEEauthorblockA{\textit{Raytheon Technologies Research Center} \\
East Hartford, Connecticut, USA \\
amit.surana@rtx.com}
\and
\IEEEauthorblockN{Tuhin Sahai}
\IEEEauthorblockA{\textit{SRI International} \\
Menlo Park, California, USA \\
tuhin.sahai@sri.com}
}

\maketitle

\begin{abstract}
In this paper, we propose efficient quantum algorithms for solving nonlinear stochastic differential equations (SDE) via the associated Fokker-Planck equation (FPE). We discretize the FPE in space and time using two well-known numerical schemes, namely Chang-Cooper and implicit finite difference. We then compute the solution of the resulting system of linear equations using the quantum linear systems algorithm.  We present detailed error and complexity analyses for both these schemes and demonstrate that our proposed algorithms, under certain conditions, provably compute the solution to the FPE within prescribed $\epsilon$ error bounds with polynomial dependence on state dimension $d$. Classical numerical methods scale exponentially with dimension, thus, our approach, under the aforementioned conditions, provides an \emph{exponential speed-up} over traditional approaches.
\end{abstract}


\section{Introduction}
\label{sec: intro}

Nonlinear stochastic differential equations (SDEs) are a popular framework for modeling real-world systems in the presence of uncertainty and stochasticity. SDEs have been used to model population growth and metabolic networks in mathematical biology, pricing of options and portfolio optimization in finance, weather, electrical networks, and even economic markets~\cite{oksendal2003stochastic, Rodkina2011}. Despite tremendous progress in state-of-the-art (SOA) algorithms for solving these equations on classical platforms, SOA methods are unable to address fundamental challenges related to curse-of-dimensionality, highly nonlinear dynamics, stiffness, and strongly coupled degrees of freedom over multiple length- and time- scales. Consequently, this limits the utility of these models for design, analysis, verification in real-world scenarios. Emerging computing platforms such as quantum computers present a unique opportunity to address these challenges. Specifically, quantum algorithms, by exploiting superposition and entanglement can provide polynomial- to exponential- acceleration over their classical counterparts. Consequently, quantum algorithms can potentially provide a new path to deal with some of these drawbacks.

Although quantum computers were originally envisioned for simulating quantum systems~\cite{feynman2018simulating}, recent work, in the area of algorithms, has demonstrated that they may be useful for simulating classical systems as well. For deterministic linear differential equations, it has been demonstrated that quantum algorithms offer the prospect of rapidly characterizing the solutions of high-dimensional systems of linear ODEs ~\cite{berry2014high,berry2017quantum,childs2020quantum}, and PDEs ~\cite{childs2021high,costa2019quantum,linden2020quantum,montanaro2016quantum}. These algorithms produce a quantum state proportional to the solution of a sparse (or block-encoded) $\nx$-dimensional system of linear differential equations in time $\textsf{poly}(\log(\nx))$), by using either the quantum linear system algorithm (QLSA) or direct Hamiltonian simulation. For deterministic nonlinear differential equations, a variety of different frameworks have been explored. For ODEs with quadratic polynomial nonlinearities, a quantum algorithm was proposed which simulates the system by storing multiple copies of the solution~\cite{leyton2008quantum}. The complexity of this approach is $\textsf{poly}(\log(\nx))$) in dimension but exponential in the evolution time $T$, scaling as $\mathcal{O}(\epsilon^{-T})$ (where, $\epsilon$  is allowed error in the solution). This is consequence of the requirement that one has to use an exponentially increasing number of copies to accurately capture the nonlinearity. To alleviate such challenges, techniques based on linear representation of dynamical systems are being extensively investigated. For example, by applying Carleman linearization it was shown that this exponential time dependence can be reduced to polynomial scaling in $T$ in certain settings \cite{liu2021efficient}, thereby providing an exponential improvement for those settings. Here the approach is based on applying truncated Carleman linearization to quadratic ODEs to transform them into a finite set of linear ODEs, and solving them using a combination of forward Euler numerical method and QLSA. This Carleman linearization based framework was recently extended for polynomial ODEs of arbitrary (finite) degree \cite{surana2022carleman}. Alternate linear representation techniques for dynamical systems such as Koopman-von Neumann mechanics, Liouville approaches ~\cite{jin2022time,joseph2020koopman,lin2022koopman}, and the Koopman framework for ergodic dynamical systems ~\cite{giannakis2022embedding} have also been proposed for simulating nonlinear ODEs on quantum platforms. For a comparison of advantages and disadvantages of these approaches, we refer the reader to ~\cite{jin2022time,lin2022koopman}.

There is, however, limited work in developing quantum algorithms for simulating nonlinear SDEs or the closely related Fokker-Planck equation (FPE) \cite{evans2012introduction,risken1996fokker}. For example, recently authors in \cite{jin2022quantum} developed a Hamiltonian simulation based quantum algorithm for solving the FPE using a ``Schrodingerisation" approach which uses a novel warped phase transformation combined with Fourier spectral discretization. However, this framework is applicable only to gradient vector fields with periodic boundary conditions. In this paper, we develop an efficient quantum algorithm to solve FPE under more general condtions. Specifically, we consider a SDE driven by Brownian motion,
\begin{equation}\label{eq: sde}
dX_t = \mu(X_t,t)dt + \sigma(X_t,t) dW_t, \qquad X_{t_0} = X_0,
\end{equation}
where the state $X_t \in \mathbb{R}^d$ evolves according to a deterministic vector field $\mu(X_t,t) \in \mathbb{R}^d \rightarrow \mathbb{R}^d $ (also known as the drift term), and is subject to random increments proportional to a multi-dimensional Wiener process $dW_t$ with independent components \cite{evans2012introduction}. We assume that the diffusion matrix $\sigma(X_t, t): \mathbb{R}^{d \times l}, d \leq l$ is full rank. 
The ``solution” to this SDE, is the probability density function (PDF) $\rho(x,t)$, whose evolution is governed by the FPE~\cite{risken1996fokker}. Although the FPE formulation is linear, and elegant and simple to state, it is well known that the methods for computing solutions to the problem suffer from a curse of dimensionality. The construction of classical algorithms for approximating the solution of the FPE in high dimensions has been an active area of research for the last few decades (for example see~\cite{chen2017beating,sun2014numerical,pmlr-v145-zhai22a}).

Our approach exploits the exponential speed-up afforded by the QLSA to address this inherent curse of dimensionality. In our proposed approach, we start by discretizing the Fokker-Planck equation in space and time using two well-known numerical schemes, namely the Chang-Cooper and implicit or backward Euler finite difference schemes. We then pose the problem of solving the resulting linear difference equations via the quantum linear systems algorithm (QLSA)~\cite{berry2017quantum}. We, consequently, call the two approaches the quantum linear systems Chang--Cooper algorithm (QLSCCA) and quantum linear systems finite difference algorithm (QLSFDA), respectively. QLSCCA and QLSFDA are both second order accurate. Additionally,  QLSCCA guarantees conservativeness and positivity of the solution. We provide detailed error and complexity analysis, proving that both the schemes, under certain conditions, can provably generate the solution to the FPE within prescribed $\epsilon$ error with polynomial scaling with respect to the state dimension $d$. Note that SOA classical numerical approaches with competing error tolerances exhibit exponential dependence on $d$.

The rest of the paper is organized as follows. In \Cref{sec: fp}, we discuss linear representation of SDEs using the FPE formulation. In \Cref{sec: qlsa_approach} and \Cref{sec:FD}, we develop our QLSCCA and QLSFDA schemes, respectively. We also analyze the associated error along with the query and gate complexities of the algorithm. We conclude and discuss future directions in \Cref{sec:conc}.

\section{The Fokker-Planck Equation}
\label{sec: fp}

The state of the stochastic process governed by (\ref{eq: sde}) can be characterized by the shape of its statistical distribution represented by the PDF. The evolution of the associated PDF is given by the FPE \cite{evans2012introduction,risken1996fokker},

\begin{equation}\label{eq: fp}
\begin{split}
\frac{\partial }{\partial t} \rho(x,t) &= - \sum_{i=1}^d \frac{\partial}{\partial x_i}  [\mu(x,t) \rho(x,t)] \\
&+ \sum_{i=1}^d\sum_{j=1}^d \frac{\partial^2}{\partial x_i \partial x_j}[D_{ij}(x,t) \rho(x,t)], \\
\rho(x,0) &= \delta(x-x_0).
\end{split}
\end{equation}
This linear partial differential equation is known as the Perron-Frobenius representation of (\ref{eq: sde}).
Note that, assuming $\Omega$ is a bounded subset in $\mathbb{R}^d$, the independent variables for the FPE lie in the domain $(x,t) \in \Omega \times (0, T)$,  $\rho(x,t)$ is the PDF, $\mu(x, t)$ is the SDE drift, and $D(x,t)$ is the diffusion tensor given by $D = \sigma \sigma^T / 2$. As a consequence of the assumption that $\sigma$ is a full rank matrix,  $D$ must be symmetric positive definite.

The FPE has been used in a wide range of applications in continuous~\cite{risken1996fokker} and discontinuous~\cite{sahai2013uncertainty} settings. Depending on the application, the boundary conditions for the FPE can take different forms. Let $S=\partial \Omega$ denote the boundary of $\Omega$. Common boundary conditions include absorbing ($\rho(S,t)=0$), reflecting (net probability flow across the boundary is zero), periodic $\rho(S^{+},t)=\rho(S^{-},t)$, and boundaries at infinity ($\displaystyle\lim_{x\to\infty} \rho(x,t)=0$). We refer the reader to~\cite{gardiner1985handbook} for descriptions of the various boundary conditions and associated properties.

A range of numerical methods for the FPE have been developed over the years. Numerical methods for simulating the FPE typically include explicit and implicit finite difference schemes, and finite element methods~\cite{pichler2013numerical}.  In what follows, we extend the Chang-Cooper \cite{chang1970practical,cc} and a second order implicit finite difference schemes for quantum settings.

\section{Chang-Cooper Scheme Based Quantum Algorithm}
\label{sec: qlsa_approach}

\subsection{Chang-Cooper Backward Euler Scheme}
\label{subsec: cc-scheme}

The FPE differs from a classic parabolic problem because of the following two constraints on the PDF; (1) Positivity: $\rho(x,t)\geq 0 \quad \forall (x,t) \in \Omega \times (0,T)$ (2) Conservativeness: $\int_{\Omega} \rho(x,t)dx = 1 \quad \forall t \in (0,T)$. Hence, a method that satisfies these constraints is expected to be more accurate.

The Chang-Cooper scheme requires that \cref{assumption1} given below are satisfied ~\cite{chang1970practical,cc}. For example, these assumptions hold true in plasma physics applicaitons ~\cite{chang1970practical}. We use the backward Euler method for time discretization for stability. We refer to this scheme as CC-BDF and is described below.  For the CC-BDF scheme, instead of using~(\ref{eq: fp}), it is more natural to consider the flux form for the FPE,
\begin{align}\label{eq:flux_fpe}
\frac{\partial }{\partial t} \rho(x,t)  = \nabla \cdot F(x,t),
\end{align}
where the flux in the \textit{i}-th direction is given by,
\begin{align}
F^i(x,t) &= M^i(x,t)\rho(x,t) + \sum_{j=1}^d D_{ij}(x,t) \frac{\partial }{\partial x_j} \rho(x,t), \\
M^i(x,t) &= \sum_{j=1}^d  \frac{\partial }{\partial x_j}D_{ij}(x,t) - \mu_i(x,t).\label{def:M}
\end{align}
Note that eqs~(\ref{eq: fp}) and~(\ref{eq:flux_fpe}) are exactly equivalent and we use the flux formulation for compactness of notation. We now summarize our assumptions.
\begin{assumption}\label{assumption1}
Assume that
\begin{itemize}
  \item $M^i(x,t)$ is a positive function.
  \item There exists a constant $\gamma > 0$ such that $\sum_{i=1}^d |M^i(x+h,t) - M^i(x,t)| \leq \gamma h \quad \forall x \in \Omega$.
  \item $M^i(x,t)$ has compact support, i.e., $M^i(x,t)=0 \quad \forall x \notin \Omega$.
\end{itemize}
\end{assumption}
\begin{assumption}\label{assumption2}
For simplicity, we assume that the diffusion matrix $D$ is diagonal. Since $D$ is positive definite, $D_{ii}(x,t)>0$.
\end{assumption}
With above assumption we can simplify the flux expression as,
\begin{align*}
F^i(x,t) &= M^i(x,t)\rho(x,t) +  D^i(x,t) \frac{\partial }{\partial x_j} \rho(x,t), \\
M^i(x,t) &=  \frac{\partial }{\partial x_i}D_{ii}(x,t) - \mu_i(x,t) \\
D^i(x,t) &= D_{ii}(x,t).
\end{align*}

Without loss of generality, consider the domain $\Omega = [0, L]^d$. We consider a uniform mesh with mesh size $h$ and time step $\Delta t$. Let $j = (j_1, j_2, \cdots, j_d)$ be a multi-index for the spatial position, $e_i$ be the unit vector in the \textit{i}-th direction,  $x_j = jh$, and $t^n = n \Delta t$ (where $n$ is the time index). Then the CC-BDF scheme is as follows,
\begin{equation}
    \label{eq: cc-bdf}
\frac{\rho_j^{n+1} - \rho_j^n}{\Delta t} = \frac{1}{h} \sum_{i=1}^d F^{i,n}_{j+e_i/2} - F^{i,n}_{j-e_i/2}\,,
\end{equation}
where,
\begin{align*}
F^{i,n}_{j+e_i/2} &= \bigg( (1-\delta_j^{i,n})M^{i,n}_{j+e_i/2} + \frac{1}{h} D^{i,n}_{j+e_i/2}\bigg)\rho_{j+e_i}^{n+1} \\
& \quad - \bigg( \frac{1}{h} D^{i,n}_{j+e_i/2}-\delta_j^{i,n} M^{i,n}_{j+e_i/2} \bigg)\rho_j^{n+1}\\
\delta_j^{i,n} &= \frac{1}{w_j^{i,n}} - \frac{1}{\exp w_j^{i,n} -1}, \quad w^{i,n}_j = h \frac{M^{i,n}_{j+e_i/2}}{ D^{i,n}_{j+e_i/2}}.
\end{align*}
Zero-flux boundary conditions are required to guarantee conservativeness of the scheme. Therefore, $F^{i,n}_{j^i_{-1/2}} = 0$ and  $F^{i,n}_{j^i_{N+1/2}} = 0$ where $j^i_k = (j_1, j_2, \cdots, j_{i-1},k, j_{i+1}, \cdots, j_d)$. Using this, we get the following conditions on the values of $\rho^{n+1}$ at the boundary,
\begin{align}
\label{eq: cc-bdf bdy}
\rho^{n+1}_{j^i_0} \exp w^{i,n}_{j^i_{-1}} &= \rho^{n+1}_{j^i_{-1}}, \quad \rho^{n+1}_{j^i_N} \exp (-w^{i,n}_{j^i_{N}}) = \rho^{n+1}_{j^i_{N+1}}.
\end{align}

\begin{theorem}\label{thm:cc}
\cite{cc_analysis} If $\Delta t \leq \frac{1}{2\gamma}$, the discretization scheme preserves positivity, conservativeness, stability, and converges with an error of order $\mathcal{O}(dh^2 + \Delta t)$.
\end{theorem}
We refer the readers to~\cite{cc_analysis} for the proof.

\subsection{Quantum Linear Systems Chang--Cooper Algorithm}
\label{subsec: QLSA}

We now describe and analyze our novel quantum linear systems Chang--Cooper algorithm (QLSCCA). We can rewrite the discretization scheme given in  (\ref{eq: cc-bdf}) in matrix form as follows,
\begin{equation}\label{eq: matrix_form}
\rho^{n} = A^n \rho^{n+1},
\end{equation}
where $A^n$ is an $(N+1)^d \times (N+1)^d$ matrix with the elements,
\begin{align} \label{eq: matrix_entries}
A^n_{pq} =
\begin{cases}
&-\alpha^{i,n}_p, \quad q=p+(N+1)^{i-1},\,\,i = 1, \dots, d,\\
&\beta^{n}_p, \quad\quad\,\,  q=p,\\
&-\gamma^{i,n}_p, \quad q=p-(N+1)^{i-1},\,\,i = 1, \dots, d\\
&0, \quad \text{otherwise}
\end{cases}
\end{align}
where,
\begin{align}\label{eq:weights}
\alpha^{i,n}_p &= \frac{\Delta t}{h^2} D^{i,n}_{j+e_i/2}W_j^{i,n}\exp w_j^{i,n}, \nonumber\\
\beta^n_p &= 1+ \frac{\Delta t}{h^2} \sum_{i=1}^d (D^{i,n}_{j+e_i/2} W_j^{i,n}\nonumber \\
& \qquad \qquad \qquad + D^{i,n}_{j-e_i/2} W^{i,n}_{j-e_i} \exp w_{j-e_i}^{i,n} ), \nonumber\\
\gamma^{i,n}_p &= \frac{\Delta t}{h^2} D^{i,n}_{j-e_i/2} W_{j-e_i}^{i,n}, \nonumber \\
W_j^{i,n} &= \frac{w_j^{i,n}}{(\exp w_j^{i,n} -1 )}.
\end{align}
and $j$ is the corresponding multi-index to $p$, i.e., $p = \sum_{i=1}^d j_i(N+1)^{i-1}$. The rows corresponding to the points at the boundary of the domain are modified using (\ref{eq: cc-bdf bdy}). Consider an arbitrary point on the boundary where row $m$ corresponds to a point $b = (b_1, b_2, \dots, b_d)$ such that $b_{i_1} = b_{i_2} = \dots , b_{i_{d_1}} = 0$ and $b_{j_1} = b_{j_2} = \dots = b_{j_{d_2}}= N$ (where $\{i_1,i_2,\hdots, i_{d_1}\}\cup\{j_1,j_2,\hdots,j_{d_2}\}\subseteq\{1,2,\hdots,d\}$), then we modify $\beta^n_m$ as,
\begin{align} \label{eq: bdy}
\beta^n_m &= \beta^n_m - \sum_{k=1}^{d_1} \gamma^{i_k,n}_m \exp w_{m-e_{i_k}}^{i_k,n} \nonumber - \sum_{l=1}^{d_2} \alpha^{i_l,n}_m \exp (-w^{i_l,n}_m)
\nonumber \\
&=\beta^n_m - \frac{\Delta t}{h^2} \sum_{k=1}^{d_1} D^{i_k,n}_{m-e_{i_k}/2} W_{m-e_{i_k}}^{i_k,n} \exp w^{i_k,n}_{m-e_{i_k}} \nonumber \\
& \qquad \quad -\frac{\Delta t}{h^2} \sum_{l=1}^{d_2} D^{i_l,n}_{m+e_{i_l}/2} W_{m}^{i_l,n}.
\end{align}

We can solve
(\ref{eq: matrix_form}) at every time step $n=0,1,...N_t-1$. Equivalently, we can vertically concatenate $\rho^n$ into a vector $\rho = \begin{bmatrix}\rho^1 ; \rho^2 ; \cdots ; \rho^{N_t} \end{bmatrix}$ and solve the following linear system,
\begin{equation}
\label{eq: qlsa_lin_sys}
L \rho = f,
\end{equation}
where,
\[
L = \begin{bmatrix}
A^0 \\
-I & A^1 &  & \makebox(0,0){\text{\Large0}}\\
& -I & A^2 & \\
& \makebox(0,0){\text{\Large0}} & \ddots & \ddots \\
& & &  -I & A^{N_t-1}
\end{bmatrix}, \quad f = \begin{bmatrix}
\rho_0 \\ 0 \\ 0 \\ \vdots \\ 0
\end{bmatrix}.
\]
Since $L$ is not a Hermitian matrix in general, we solve the system of equations by using a dilated Hermitian matrix of the form,
\[
\Tilde{L} = \begin{bmatrix}
0 & L^H \\
L & 0 \\
\end{bmatrix},
\]
which enlarges the matrix dimension by a factor of 2. Using the QLSA, we produce a normalized solution state $|L^{-1}f\rangle$ within error $\epsilon$. While there are many variants of QLSA, we use the QLSA proposed in~\cite{qlsa_optimal} which has the best known query/gate complexity as stated below.
\begin{theorem}\label{thm:qlsa}
\cite{qlsa_optimal}
Let $Ax=b$ be a system of linear equations, where $A$ is an $N\times N$ matrix with sparsity $s$ and condition number $\kappa$. Given an oracle, that computes the locations and values of non-zero entries of operator $A$ and an oracle that prepares $|b\rangle$, there exists a quantum algorithm (referred to as QLSA) that produces the normalized state $|A^{-1}b\rangle$ to within error $\epsilon$  in terms of the $l_2$ norm, using an average number of oracle calls,
\[
\mathcal{O}(s \kappa \log \frac{1}{\epsilon}).
\]
\end{theorem}
\begin{proof}
Please see~\cite{qlsa_optimal} for the detailed proof.
\end{proof}
Thus, the query complexity of the QLSA is $\mathcal{O}(s \kappa \log \frac{1}{\epsilon})$, and the gate complexity is larger than the query complexity only by a logarithmic factor~\cite{qlsa_optimal}.

We apply the QLSA to, now
\begin{equation}
\label{eq: qlsa_lin_sys_exd}
L_e \rho = f_e,
\end{equation}
where,
\[
L_e = \begin{bmatrix}
A^0 \\
-I & A^1 & \\
& -I & A^2 & &  & \makebox(0,0){\text{\huge0}}\\
& & \ddots & \ddots \\
&  \makebox(0,0){\text{\huge0}} & &  -I & A^{N_t-1}  \\
& & &  & -I & I  \\
& & & & \ddots & \ddots \\
& & & & & & -I & I
\end{bmatrix}, \]
$\rho_e = \begin{bmatrix}\rho^1 ; \rho^2 ; \cdots ; \rho^{N_t} ; \rho^{N_t} ; \cdots ; \rho^{N_t} \end{bmatrix}$ and $f_e = \begin{bmatrix}
\rho_0 ; 0 ; 0; \cdots ; 0
\end{bmatrix}$. The condition number of the extended system will be of the same order of magnitude and hence will not affect the query and gate complexity discussed later. The extended system can be expressed in quantum form as,
\begin{equation}\label{eq: qlsa_lin_sys_exdq}
\mathcal{L}_e|\rho_e\rangle = |f_e\rangle,
\end{equation}
where,
\begin{equation}
\mathcal{L}_e=\sum_{j=0}^{N_t-1}|j\rangle \langle j|\otimes A^j+\sum_{j=N_t}^{2N_t-1}|j\rangle \langle j|\otimes I-\sum_{j=1}^{2N_t-1}|j\rangle\langle j-1|\otimes I, \label{eq:L}
\end{equation}
\begin{equation}
|\rho_e\rangle=\sum_{j=0}^{2N_t-1}\rho^{j+1}|j\rangle,\text{ and } |f_e\rangle=\frac{1}{\|\rho_0\|}\rho_0 |0\rangle\notag.\label{eq:fe}
\end{equation}
Here, we have used the standard ``bra" $|\cdot\rangle$ and ``ket" $\langle \cdot |$ notation to represent a quantum state and its conjugate transpose, respectively \cite{nielsen2002quantum}.
Finally, the steps involved in our proposed quantum algorithm, which we refer to as QLSCCA, are as follows:
\begin{itemize}
  \item For given $\epsilon$, choose $h,\Delta t$ based on \Cref{thm:err} and prepare $|f_e\rangle$ (\ref{eq:fe}) by an appropriate unitary transformation and $\mathcal{L}_e$ (\ref{eq:L}) by block encoding.
  \item Apply QLSA algorithm to the linear system~(\ref{eq: qlsa_lin_sys_exdq}) and obtain the approximate solution,
  \[
   |\rho_q\rangle=\sum_{j=0}^{2N_t-1} \rho^j_q|j\rangle.
  \]
  \item Post-select on $j$ for some $j\in\{N_t,N_t+1,\cdots,2N_t\}$ to obtain the approximate normalized solution $\frac{\rho^{N_t}_q}{\|\rho^{N_t}_q\|}$ of the Fokker-Planck equation at time $T$. We note that, in our approach, $\rho^{N_t}$ is replicated $N_t$ times to increase the probability of measuring $\rho^{N_t}$.
\end{itemize}
Using the extended system~(\ref{eq: qlsa_lin_sys_exd}) with the post-selection, QLSCCA produces the desired solution with a constant $\Omega(1)$ probability. Next, we present detailed error and complexity analysis of the QLSCCA.

\subsection{Error Analysis of QLSCCA}

\begin{theorem}\label{thm:err}
Let $\rho(T)$ be the solution of the FPE~(\ref{eq: fp}) and $\frac{\rho^{N_t}_q}{\|\rho^{N_t}_q\|}$ be the solution obtained by the QLSCCA, then for every $0\leq \epsilon \leq 1$, there exist choices for $h$ and $\Delta t$, such that,
\begin{equation}\label{err1}
\bigg\|\frac{\rho(T)}{\|\rho(T)\|}-\frac{\rho^{N_t}_q}{\|\rho^{N_t}_q\|}\bigg\|\leq \epsilon.
\end{equation}
\end{theorem}
\begin{proof}
Let $\rho^{N_t}_{q}$ be the approximate solution of~(\ref{eq: qlsa_lin_sys}) obtained by QLSA, such that,
\begin{equation}\label{errQ}
\|\frac{\rho^{N_t}_c}{\|\rho^{N_t}_c\|}-\frac{\rho^{N_t}_q}{\|\rho^{N_t}_q\|}\|\leq \epsilon/2.
\end{equation}
where, $\rho^{N_t}_{c}$ is the exact solution of the discretized system of equations given in~(\ref{eq: qlsa_lin_sys}). By \Cref{thm:cc}, for $\Delta t\leq \frac{1}{2\gamma}$, $\rho^{N_t}_{c}$ satisfies,
\begin{equation}\label{errCC}
\|\rho(T)-\rho^{N_t}_{c}\|\leq C(\Delta t+dh^2),
\end{equation}
where, without loss of generality, we assume $C>1$.
We next show how to select $h$  and $\Delta t$, such that,
\begin{equation}\label{errCC1}
\frac{\|\rho(T)-\rho^{N_t}_{c}\|}{\|\rho(T) \|}\leq \frac{C(\Delta t+dh^2)}{\|\rho(T) \|}\leq \epsilon/4.
\end{equation}
We consider two cases,
\begin{itemize}
  \item If $\gamma \geq \frac{4C}{\epsilon \|\rho(T)\|}$, choose $\Delta t = \frac{1}{2\gamma}$ and $d\,h^2 = \frac{\epsilon \|\rho(T) \|}{4C} - \frac{1}{2\gamma} $.
  \item Otherwise, choose $\Delta t = dh^2 = \frac{\epsilon \|\rho(T) \|}{8C\gamma}$.
\end{itemize}
Since $\|\rho(T) \| \equiv \|\rho(T) \|_2\leq \|\rho(T) \|_1=1$, in either case, it follows that $\Delta t\leq \frac{1}{2\gamma}$. Thus, by using~(\ref{errCC}) we conclude that inequality (\ref{errCC1}) holds.

Furthermore, since,
\begin{equation}
\|\rho(T)\|=\|\rho(T)+\rho^{N_t}_{c}-\rho^{N_t}_{c}\|\leq \|\rho^{N_t}_{c}\|+\|\rho(T)-\rho^{N_t}_{c}\|,\notag
\end{equation}
we get,
\begin{equation}
\frac{\|\rho^{N_t}_{c}\|}{\|\rho(T)\|}\geq \frac{\|\rho(T)\|-\|\rho(T)-\rho^{N_t}_{c}\|}{\|\rho(T)\|}\geq 1-\frac{\|\rho(T)-\rho^{N_t}_{c}\|}{\|\rho(T) \|},\notag
\end{equation}
and finally using~(\ref{errCC1}), we obtain
\begin{equation}\label{eq:rat2}
\bigg\|1-\frac{\|\rho^{N_t}_{c}\|}{\|\rho(T)\|}\bigg \|\leq \frac{\|\rho(T)-\rho^{N_t}_{c}\|}{\|\rho(T) \|}\leq \epsilon/4.
\end{equation}
Thus,
\begin{eqnarray}\label{eq:inq}
&&\bigg\|\frac{\rho^{N_t}_c}{\|\rho^{N_t}_c\|}-\frac{\rho(T)}{\|\rho(T)\|}\bigg \| \notag \\ 
&\leq&\bigg\|\frac{\rho^{N_t}_c}{\|\rho^{N_t}_c\|}-\frac{\rho^{N_t}_c}{\|\rho(T)\|}\bigg\|+\bigg\|\frac{\rho^{N_t}_c-\rho(T)}{\|\rho(T)\|}\bigg\|\notag\\
&\leq & \epsilon/4+\epsilon/4=\epsilon/2.
\end{eqnarray}
Finally, using~(\ref{errQ}) and~(\ref{eq:inq}),
\begin{eqnarray}\label{errp}
&&\bigg\|\frac{\rho(T)}{\|\rho(T)\|}-\frac{\rho^{N_t}_q}{\|\rho^{N_t}_q\|}\bigg\| \notag \\
&\leq& \bigg\|\frac{\rho(T)}{\|\rho(T)\|}-\frac{\rho^{N_t}_c}{\|\rho^{N_t}_c\|}\bigg\|+\bigg\|\frac{\rho^{N_t}_c}{\|\rho^{N_t}_c\|}-\frac{\rho^{N_t}_q}{\|\rho^{N_t}_q\|}\bigg\|\notag\\
&\leq&\epsilon/2+\epsilon/2=\epsilon,
\end{eqnarray}
we obtain the desired result.
\end{proof}

\subsection{Complexity Analysis of QLSCCA}
\label{subsec: cc-complexity}
We present the query and gate complexity of the algorithm for solving~(\ref{eq: qlsa_lin_sys}) in~\Cref{thm: complexity}. The proof for this theorem relies on~\Cref{lem: SDD} and~\Cref{cor: Ainv}. The complexity results are stated in terms of number of accesses to oracles, analogous to~\Cref{thm:qlsa}.

\begin{lemma} \label{lem: SDD}
If $\Delta t < \frac{1}{\gamma}$, matrix $A^n$ in (\ref{eq: matrix_form}) is strictly diagonally dominant (SDD) in both rows and columns.
\end{lemma}

\begin{proof}
To show that $A^n$ is SDD in columns, we prove that $|A^n_{pp}|- \sum_{q\neq p} |A^n_{qp}|> 0, \quad \forall p\in [1,(N+1)^d]$. Consider the column corresponding to any interior point $j$ in the mesh,
\begin{flalign*}
&|A^n_{pp}|- \sum_{q\neq p} |A^n_{qp}| \\
&= 1+ \frac{\Delta t}{h^2} \sum_{i=1}^d (D^{i,n}_{j+e_i/2} W_j^{i,n} + D^{i,n}_{j-e_i/2} W^{i,n}_{j-e_i} \exp w_{j-e_i}^{i,n} )  \\
&- \frac{\Delta t}{h^2}  \sum_{i=1}^{d} ( D^{i,n}_{j-e_i/2}W_{j-e_i}^{i,n}\exp w_{j-e_i}^{i,n} -   D^{i,n}_{j+e_i/2} W_{j}^{i,n}) \\
&= 1 > 0.
\end{flalign*}
The same result also holds true at the boundary points and can be shown using (\ref{eq: matrix_entries}) and (\ref{eq: bdy}).

Similarly, to show that $A$ is SDD in rows, consider the row corresponding to any interior point $j$ in the mesh,
\begin{flalign*}
&|A^n_{pp}|- \sum_{q\neq p} |A^n_{pq}| \\
&= 1+ \frac{\Delta t}{h^2} \sum_{i=1}^d (D^{i,n}_{j+e_i/2} W_j^{i,n} + D^{i,n}_{j-e_i/2} W^{i,n}_{j-e_i} \exp w_{j-e_i}^{i,n} ) \\
& - \frac{\Delta t}{h^2}  \sum_{i=1}^{d} ( D^{i,n}_{j+e_i/2}W_{j}^{i,n}\exp w_{j}^{i,n}+   D^{i,n}_{j-e_i/2} W_{j-e_i}^{i,n}) \\
&= 1-  \frac{\Delta t}{h^2} \sum_{i=1}^d (D^{i,n}_{j+e_i/2} w_j^{i,n} - D^{i,n}_{j-e_i/2} w^{i,n}_{j-e_i} )  \\
&= 1-  \frac{\Delta t}{h} \sum_{i=1}^d (M^{i,n}_{j+e_i/2} - M^{i,n}_{j-e_i/2} )\\
&\geq 1 - \gamma \Delta t > 0.
\end{flalign*}
using the definition of $W_j^{i,n}$ and Lipschitz continuity of $M^i(x,t)$. To show that the same result also holds at the boundary, we need the assumption that $M^i(x,t)$ has compact support. Consider an arbitrary point $b = (b_1, b_2, \dots, b_d)$ on the boundary such that $b_{i_1} = b_{i_2} = \dots , b_{i_{d_1}} = 0$ and $b_{j_1} = b_{j_2} = \dots = b_{j_{d_2}}= N$ (where $\mathcal{I} \coloneqq \{i_1,i_2,\hdots, i_{d_1}\},  \mathcal{J} \coloneqq \{j_1,j_2,\hdots,j_{d_2}\} \text{, and } \mathcal{I} \cup \mathcal{J}\subseteq\{1,2,\hdots,d\}$).
Then using (\ref{eq: bdy}) and following the same procedure, we can show
\begin{flalign*}
&|A^n_{pp}|- \sum_{q\neq p} |A^n_{pq}| \geq 1 - \gamma \Delta t > 0
\end{flalign*}
where we have used the compact support assumption on $M^{i,n}, M^{j,n}$.
\end{proof}

\begin{corollary} \label{cor: A}
$\|A^n \|_2 \leq 2\,\frac{\Delta t}{h^2} (C^n \, h\,d +2d) +1$, where $C^n>0$ is a time-dependent constant.
\end{corollary}
\begin{proof}
From~\Cref{lem: SDD}, $A^n$ is SDD in rows and columns. Then,
\begin{align*}
\|A^n\|_\infty &= \max_p \left(|A^n_{pp}| + \sum_{q\neq p}|A^n_{qp}|\right) \nonumber \\
&\leq  2\, \max_p |A^n_{pp}| -1, \\
\|A\|_1 &= \max_p \left(|A^n_{pp}| + \sum_{q\neq p}|A^n_{pq}|\right) \nonumber \\
&= 2\,\max_p |A^n_{pp}| -1,
\end{align*}

Define $Q^n$ as,
\begin{align*}
Q^n &= \max_j \sum_{i=1}^d (D^{i,n}_{j+e_i/2} W_j^{i,n} + D^{i,n}_{j-e_i/2} W^{i,n}_{j-e_i} \exp w_{j-e_i}^{i,n} )\\
&\leq \sum_{i=1}^d \Big(\frac{C_1^{i,n}h}{\exp (C_2^{i,n} h)-1} + \frac{C_3^{i,n} h \exp (C_4^{i,n} h)}{\exp (C_4^{i,n} h)-1}\Big)\\
& \leq C^{n}\,d\, h \, \frac{\exp(C^{n}h)+1}{\exp(C^nh)-1} \\
& \leq C^n\,d\, h \, \Big(1+ \frac{2}{C^nh}\Big) = d (C^nh + 2),
\end{align*}
where $C_1^{i,n}, C_2^{i,n}, C_3^{i,n}, C_4^{i,n}$ and $C^n$ are some finite positive constants that capture the maximum values of $D^{i,n}_{j+e_i/2}$ and $D^{i,n}_{j-e_i/2}$. For the first step, we use the definition in~(\ref{eq:weights}) and on the last step we use the fact that $(e^x +1)/(e^x-1) \leq (1+2/x),\,\, \forall x>0$. Then,
\begin{align*}
\|A^n\|_\infty &\leq  2\,  \frac{\Delta t}{h^2} Q^n +1 , \qquad \|A^n\|_1 \leq 2\,\frac{\Delta t}{h^2} Q^n +1, \\
\|A^n\|_2 &\leq \sqrt{\|A\|_\infty \|A\|_1} \leq 2\,\frac{\Delta t}{h^2} Q^n +1
\end{align*}
\end{proof}

\begin{corollary} \label{cor: Ainv}
$\|(A^n)\inv \|_2 \leq \frac{1}{\sqrt{1-\gamma \Delta t}}$.
\end{corollary}
\begin{proof}
From~\Cref{lem: SDD}, $A^n$ is SDD in rows and columns. Then, \[
\|(A^n)\inv\|_2 \leq \frac{1}{\sqrt{\alpha \beta}},
\]
where $\alpha = \min_p \left(|A^n_{pp}|-\sum_{q \neq p} |A^n_{qp}|\right)$ and $\beta = \min_p \left(|A^n_{pp}|-\sum_{q \neq p} |A^n_{pq}|\right)$~\cite{sigma_bound}. From the proof of~\Cref{lem: SDD}, we see that $\alpha = 1$ and $\beta \geq 1-\gamma \Delta t$.
\end{proof}


\begin{theorem} \label{thm: complexity}
For a given error tolerance $\epsilon<1$, the query complexity of QLSCCA is,
\[
\mathcal{O}\bigg(\frac{e^{\gamma T}}{\gamma} \frac{d^3}{q\epsilon}  \log \frac{1}{\epsilon}\bigg),
\]
where, $q=\|\rho(T)\|$, and it produces a normalized solution to the FPE (\ref{eq: fp}) within $\epsilon$ ($l_2$ error) with a constant probability $\Omega(1)$ of success. The gate complexity is larger than the query complexity only by a logarithmic factor.
\end{theorem}

\begin{proof}
Recall, from \Cref{thm:qlsa}, the query complexity of the QLSA algorithm is $\mathcal{O}(s \kappa \log \frac{1}{\epsilon})$. For (\ref{eq: qlsa_lin_sys}), the sparsity is $s = \mathcal{O}(d)$. We now show that the condition number $\kappa = \| L\|_2 \| L\inv \|_2 = \mathcal{O}(e^{\gamma T} d^2 / (\gamma \epsilon))$.
\begin{itemize}
\item \textit{Bound on $\|L\|_2$}
\begin{align*}
\|L\|_2 &\leq \|I\|_2 + \max_{n=1,\dots, N_t} \|A^n\|_2 \\
& \leq 2 + 2 \, \frac{\Delta t}{h^2} \Big( 2d + h \, d \max_{n=1,\dots, N_t} C^n \Big) \\
&= 2 \Big( 1+ \frac{\Delta t}{h^2} 2d + \frac{\Delta t}{h} Cd \Big),
\end{align*}
using triangle inequality, $\|I\|_2 = 1$ and \Cref{cor: A}. Here $C = \max_{n=1,\dots, N_t} C^n >0$.

\item \textit{Bound on $\|L\inv\|_2$}:
 $L\inv$ can be written exactly in a block lower-triangular form. The $(i,j)$ block of $L\inv$ for $i\geq j$ is given by,
 \[
 (L\inv)_{i,j} = \prod_{k=i-1}^{j-1} (A^k)\inv
 \]
Then, $\|L\inv\|_2$ can be bounded using a sequence of triangle inequalities and the submultiplicativity of $\|.\|_2$.
\[\|L\inv\|_2 \leq \sum_{j=1}^{N_t} \max_{i=0, \dots, N_t-1} \|(A^i)\inv\|_2^{j}\]
From \Cref{cor: Ainv}, $\|(A^n)\inv\|_2 \leq 1/\sqrt{(1-\gamma \Delta t)} < v = 1+\gamma \Delta t$ for $\gamma \Delta t < 0.5$. Therefore,
\begin{align*}
\|L\inv\|_2 < \sum_{i=1}^{N_t} v^i &= v \frac{v^{N_t}-1}{v-1}\\
&= (1+\gamma \Delta t) \frac{(1+\gamma \Delta t)^{N_t}-1}{\gamma \Delta t} \\
& < \frac{3}{2} \frac{e^{\gamma T}}{\gamma \Delta t} \,\, \big( \because(1+x/n)^{n} \leq e^x\big).
\end{align*}
\end{itemize}

Hence, we get the following bound on the condition number $\kappa$.
\[
\kappa < 3 \frac{e^{\gamma T}}{\gamma}\Big( \frac{1}{ \Delta t}+ \frac{2d}{h^2} +  \frac{Cd}{h}  \Big).
\]
Finally, using the choice of $h,\Delta t$ from \Cref{thm:err}, we get $\kappa(L) = \mathcal{O}(e^{\gamma T} d^2 / (\gamma q \epsilon))$ leading to the desired result.
\end{proof}


\begin{remark}\label{remark1}
Assuming $T=\mathcal{O}(1)$, we conclude from Theorem \ref{thm: complexity}, that the QLSCCA approach has a polynomial dependence on state dimension $d$, compared to classical CC implementation which exhibits exponential dependence on the dimension.
\end{remark}

\section{Finite Difference Scheme Based Quantum Algorithm}\label{sec:FD}
Similar to QLSCCA, we construct a quantum algorithm based on an implicit finite difference scheme. 
Recall that for the CC-BDF scheme a positivity condition was imposed on $M^i(x,t)$ (see Assumption \ref{assumption1}), which can be restrictive for general applications. We relax this assumption by considering an implicit finite difference scheme. This scheme requires a different set of assumptions as summarized in Assumption \ref{assumption3}, which are less restrictive. However, a drawback of this scheme is that particle conservation is no longer satisfied even if approximate particle conserving boundary conditions are specified~\cite{cc}.

\subsection{Implicit Finite Difference Scheme}
To develop a finite difference scheme, we first rewrite (\ref{eq: fp}) as follows,
\begin{align}
&\frac{\partial }{\partial t} \rho(x,t) =  \sum_{i=1}^d \frac{\partial}{\partial x_i}  \bigg[\bigg(\frac{\partial D^i(x,t)}{\partial x_i}-\mu(x,t)\bigg) \rho(x,t)\bigg] \nonumber \\
&\quad  + \sum_{i=1}^d \frac{\partial}{\partial x_i}\bigg[D^i(x,t) \frac{\partial \rho(x,t)}{\partial x_i} \bigg] \nonumber \\
&\quad = \sum_{i=1}^d \frac{\partial M^i(x,t)}{\partial x_i} \rho(x,t)  + \sum_{i=1}^d D^i(x,t) \frac{\partial^2 \rho(x,t)}{\partial x_i^2} \nonumber \\
&\quad +\sum_{i=1}^d \bigg(M^i(x,t)+ \frac{\partial D^i(x,t)}{\partial x_i} \bigg) \frac{\partial \rho(x,t)}{\partial x_i}, \label{eq: fpfd}
\end{align}
where, $M^{i}(x,t)$ is as defined in (\ref{def:M}) and we have assumed that the diffusion matrix $D$ is diagonal as before (see assumption~\ref{assumption2}).

Without loss of generality, we again consider a finite domain $\Omega = [0, L]^d$. We use a central finite difference scheme for spatial discretization and a backward Euler scheme for time discretization of above equation, and refer to this scheme as FD-BDF. Let $h$ be the uniform mesh size along all dimensions  and $\Delta t$ be the time step. At a $d$ dimensional grid point $j = (j_1, j_2, \dots j_d)$, the FD-BDF scheme leads to,
\begin{align}
&\frac{\rho_j^{n+1} - \rho_j^{n}}{\Delta t} = \sum_{i=1}^d a_j^{i,n} \rho_j^{n+1} + \sum_{i=1}^d b_j^{i,n} \frac{\rho_{j+e_i}^{n+1}- \rho_{j-e_i}^{n+1}}{2h} \notag \\
&+ \sum_{i=1}^d c_j^{i,n} \frac{\rho_{j+e_i}^{n+1} - 2 \rho_j^{n+1} + \rho_{j+e_i}^{n+1} }{h^2},\label{eq:FPEfd}
\end{align}
where,
\begin{align*}
& a_j^{i,n} = \frac{\partial M^i(x,t)}{\partial x} \bigg|_{j,n}, b_j^{i,n} = M^i_{j,n} + \frac{\partial D^i(x,t)}{\partial x} \bigg|_{j,n},\\
& c_j^{i,n} = D_j^{i,n}.
\end{align*}
This scheme is second order accurate in space and first order accurate in time. One can use analytic expressions or any discretization scheme for the coefficients in the above equation. We assume a Dirichlet boundary condition with $\rho(x,t) =0$ at the boundary of $\Omega$. We  make following additional assumptions for the rest of this section. 


\begin{assumption}\label{assumption3}
Assume that,
\begin{itemize}
  \item There exists a constant $\gamma > 0$ such that $\sum_{i=1}^d |M^i(x+h,t) - M^i(x,t)| \leq \gamma h \quad \forall x \in \Omega$.
  \item $M^i(x,t)$ is thrice differentiable with bounded derivatives in $\Omega$.
  \item $D^i(x,t)$ is differentiable four times with bounded derivatives in $\Omega$.
  \item $|b_j^{i,n}|L/ c_j^{i,n} \leq 2 $ at every grid point $j$ and time step $n$, and $i=1, \dots, d$.
\end{itemize}
\end{assumption}
Note that we have relaxed the positivity assumption (see  \Cref{assumption1}) on $M(x,t)$ and only require it to be a Lipschitz continuous function. 
The assumption on boundedness of derivatives of $M^i$ and $D^i$ implies that the diffusion matrix and the drift term are sufficiently smooth. The ratio $|b_j^{i,n}|L/ c_j^{i,n}$ can be interpreted as a generalization of Peclet number. Assume that $\mu$ and $D^i$ are constant, i.e. $\mu(x,t)\equiv u$ and $D^i(x,t)\equiv D$ , then $|b_j^{i,n}|L/ c_j^{i,n}$ reduces to a constant $uL/D$ which is the Peclet number measuring the ratio of advection transport rate to the diffusion transport rate \cite{patankar2018numerical}. 


\begin{theorem}
~\cite{implicit_fd_convg} If $\Delta t \leq \frac{1}{2\gamma}$, the FD-BDF scheme is stable and converges with an error of order $\mathcal{O}(dh^2 + \Delta t)$.
\end{theorem}

\subsection{Quantum Linear Systems Finite Difference Algorithm}
We express the discretized system~(\ref{eq:FPEfd}) in matrix form as follows,
\begin{equation}
\rho^n = A^n \rho^{n+1},\label{eq:FD}
\end{equation}
where $A^n$ is a $(N+1)^d \times (N+1)^d$ matrix with entries,
\begin{align} \label{eq: fd_matrix_entries}
A^n_{pq} =
\begin{cases}
-&\alpha^{i,n}_p, \quad q=p+(N+1)^{i-1}, i = 1, 2, \cdots, d,\\
&\beta^{n}_p, \quad  q=p,\\
-&\gamma^{i,n}_p, \quad q=p-(N+1)^{i-1}, i = 1, 2, \cdots, d,\\
&0, \quad \text{otherwise}
\end{cases}
\end{align}
where,
\begin{align*}
\alpha^{i,n}_p &= \frac{\Delta t}{h^2} c_j^{i,n} - \frac{\Delta t}{2h}b_j^{i,n}, \\
\beta^n_p &= 1- \Delta t \sum_{i=1}^d a_j^{i,n} + 2 \frac{\Delta t}{h^2} \sum_{i=1}^d  c_j^{i,n} , \\
\gamma^{i,n}_p &= \frac{\Delta t}{h^2} c_j^{i,n} + \frac{\Delta t}{2h}b_j^{i,n}.
\end{align*}
Following, the procedure in Section \ref{subsec: QLSA}, we embed (\ref{eq:FD}) into a linear system of the form of (\ref{eq: qlsa_lin_sys}),
\begin{equation}
L \rho = f,
\end{equation}
and refer to the resulting quantum algorithm (analogous to QLSCCA),  as the quantum linear systems finite difference algorithm (QLSFDA).

\subsection{Error and Complexity Analysis of QLSFDA}
\label{sec: fd_error}

Next, we present the error and complexity analysis for solving (\ref{eq:FD}) using QLSFDA. Both analyses are similar to their QLSCCA counterparts discussed in the previous section. 

\begin{theorem}\label{thm: fd_err}
Let $\rho(T)$ be the solution of the FPE~(\ref{eq: fp}) and $\frac{\rho^{N_t}_q}{\|\rho^{N_t}_q\|}$ be the solution obtained by the QLSFDA, then for every $0\leq \epsilon \leq 1$, there exist choices for $h$ and $\Delta t$, such that,
\begin{equation}\label{err_fd}
\bigg\|\frac{\rho(T)}{\|\rho(T)\|}-\frac{\rho^{N_t}_q}{\|\rho^{N_t}_q\|}\bigg\|\leq \epsilon.
\end{equation}
\end{theorem}
\begin{proof}
Let $\rho^{N_t}_{q}$ be the approximate solution of~(\ref{eq:FD}) obtained by QLSA, such that,
\begin{equation}\label{err_FD}
\bigg\|\frac{\rho^{N_t}_c}{\|\rho^{N_t}_c\|}-\frac{\rho^{N_t}_q}{\|\rho^{N_t}_q\|}\bigg\|\leq \epsilon/2,
\end{equation}
where, $\rho^{N_t}_{c}$ is the exact solution of the discretized system of equations given in~(\ref{eq: qlsa_lin_sys}). $\rho^{N_t}_{c}$ satisfies,
\begin{equation}\label{errFD}
\|\rho(T)-\rho^{N_t}_{c}\|\leq C(\Delta t+dh^2),
\end{equation}
where, without loss of generality, we assume $C>1$.
We next show how to select $h$  and $\Delta t$, such that,
\begin{equation}\label{errFD1}
\frac{\|\rho(T)-\rho^{N_t}_{c}\|}{\|\rho(T) \|}\leq \frac{C(\Delta t+dh^2)}{\|\rho(T) \|}\leq \epsilon/4.
\end{equation}
We consider two cases,
\begin{itemize}
  \item If $\gamma \geq \frac{4C}{\epsilon \|\rho(T)\|}$, choose $\Delta t = \frac{1}{2\gamma}$ and $d\,h^2 = \frac{\epsilon \|\rho(T) \|}{4C} - \frac{1}{2\gamma} $.
  \item Otherwise, choose $\Delta t = dh^2 = \frac{\epsilon \|\rho(T) \|}{8C\gamma}$.
\end{itemize}
Since $\|\rho(T) \| \equiv \|\rho(T) \|_2\leq \|\rho(T) \|_1=1$, in either case, it follows that $\Delta t\leq \frac{1}{2 \gamma}$. Thus, by using~(\ref{errFD}) we conclude that inequality (\ref{errFD1}) holds. Then, we can use a similar analysis as in \Cref{thm:err} to achieve the desired result. 

\end{proof}

The query and gate complexity of QLSFDA is presented in \Cref{thm: fd_complexity} whose proof relies on \Cref{lem: fd_SDD} and \Cref{cor: fd}.

\begin{lemma} \label{lem: fd_SDD}
The matrix  $A^n$ given in (\ref{eq: fd_matrix_entries}) is SDD both in rows and columns.
\end{lemma}
\begin{proof}
To prove that $A^n$ is SDD in rows, we show that $|A^n_{pp}|- \sum_{q\neq p} |A^n_{pq}|> 0, \quad \forall p\in [1,(N+1)^d]$. Consider the column corresponding to any interior point $j$ in the mesh, then,
\begin{align*}
&|A^n_{pp}|- \sum_{q\neq p} |A^n_{pq}| = 1- \Delta t \sum_{i=1}^d a_j^{i,n} + 2 \frac{\Delta t}{h^2} \sum_{i=1}^d  c_j^{i,n} \\
&- \sum_{i=1}^d \Big|\frac{\Delta t}{h^2} c_j^{i,n} - \frac{\Delta t}{2h}b_j^{i,n} \Big| - \sum_{i=1}^d \Big|\frac{\Delta t}{h^2} c_j^{i,n} + \frac{\Delta t}{2h}b_j^{i,n}\Big| \\
&= 1 - \Delta t \sum_{i=1}^d a_j^{i,n} \quad \Big(\because \frac{|b_j^{i,n}| h}{c_j^{i,n}} \leq \frac{|b_j^{i,n}|L}{c_j^{i,n}} \leq 2 \Big) \\
& \geq 1- \gamma \Delta t >0 \quad (\because \gamma \Delta t < 1).
\end{align*}
The same result holds true for the boundary points as well. Similarly, to establish that $A^n$ is SDD in columns, we show $|A^n_{pp}|- \sum_{q\neq p} |A^n_{qp}|> 0, \quad \forall p\in [1,(N+1)^d]$. Consider the column corresponding to any interior point $j$ in the mesh, then,
\begin{align*}
&|A^n_{pp}|- \sum_{q\neq p} |A^n_{qp}| = 1- \Delta t \sum_{i=1}^d a_j^{i,n} + 2 \frac{\Delta t}{h^2} \sum_{i=1}^d  c_j^{i,n} \\
&- \sum_{i=1}^d |\frac{\Delta t}{h^2} c_{j+e_i}^{i,n} - \frac{\Delta t}{2h}b_{j+e_i}^{i,n} | - \sum_{i=1}^d |\frac{\Delta t}{h^2} c_{j-e_i}^{i,n} + \frac{\Delta t}{2h}b_{j-e_i}^{i,n}| \\
&= 1- \Delta t \sum_{i=1}^d a_j^{i,n} + \frac{\Delta t}{2h}\sum_{i=1}^d \big (b^{i,n}_{j+e_i} -   b^{i,n}_{j-e_i})  \\
&- \frac{\Delta t}{h^2} \sum_{i=1}^d \big( c_{j+e_i}^{i,n} - 2c_j^{i,n} + c_{j-e_i}^{i,n} \big)\\
&= 1- \Delta t \sum_{i=1}^d \frac{\partial M^i(x,t)}{\partial x}\Big|_{j,n} \\
& - \frac{\Delta t}{h^2} \sum_{i=1}^d \big( D_{j+e_i}^{i,n} - 2D_j^{i,n} + D_{j-e_i}^{i,n} \big) \\
&+ \frac{\Delta t}{2h}\sum_{i=1}^d \bigg (M^{i,n}_{j+e_i} -   M^{i,n}_{j-e_i} \\
&\qquad\qquad + \frac{\partial D^i(x,t)}{\partial x}\Big|_{j+e_i,n} - \frac{\partial D^i(x,t)}{\partial x}\Big|_{j-e_i,n} \bigg)\\
&= 1 - \Delta t \sum_{i=1}^d \Bigg(\frac{\partial M^i(x,t)}{\partial x}\Big|_{j,n} - \frac{ M^{i,n}_{j+e_i} -   M^{i,n}_{j-e_i}}{2h} \Bigg)\\
& + \Delta t \sum_{i=1}^d \Bigg( \frac{1}{2h}\bigg(\frac{\partial D^i(x,t)}{\partial x}\Big|_{j+e_i,n} - \frac{\partial D^i(x,t)}{\partial x}\Big|_{j-e_i,n}\bigg) \\ 
&\qquad \qquad -\frac{D_{j+e_i}^{i,n} - 2D_j^{i,n} + D_{j-e_i}^{i,n}}{h^2} \Bigg) \\
&= 1 + c_1 \Delta t \,(dh^2), 
\end{align*}
where $c_1$ is a finite constant. The last equality follows from the assumption that the third and fourth derivatives of $M^i$ and $D^i$ are bounded, respectively, see Assumption \ref{assumption3}.


If $c_1 \geq 0$, $|A^n_{pp}|- \sum_{q\neq p} |A^n_{qp}| \geq 1$.
If $c_1 < 0$ we can always choose $dh^2 \leq \gamma / |c_1|$ which is consistent with the required conditions in \Cref{thm: fd_err}. With such a choice of $dh^2$, $|A^n_{pp}|- \sum_{q\neq p} |A^n_{qp}| \geq 1 - \gamma \Delta t > 0$. Hence, in either case,
\[
|A^n_{pp}|- \sum_{q\neq p} |A^n_{qp}| \geq 1 - \gamma \Delta t > 0.
\]
\end{proof}

\begin{corollary} \label{cor: fd}
$\|A^n \|_2 \leq 1 + 4 \frac{\Delta t}{h^2} d \, D_{\max}$ and $\|(A^n)\inv\|_2 \leq \big(1 + \gamma \Delta t\big)^2$
\end{corollary}
\begin{proof}
From the Lemma \ref{lem: fd_SDD}, $A^n$ is SDD in rows and columns. Hence,
\begin{align*}
&\|A^n\|_\infty = \max_p \left(|A^n_{pp}| + \sum_{q\neq p}|A^n_{qp}|\right) \\
&\leq 2\, \max_p |A^n_{pp}| -1 + \gamma \Delta t \leq 1 - \gamma \Delta t +  4 \frac{\Delta t}{h^2} d D_{\max}  \\
& \leq 1 +  4 \frac{\Delta t}{h^2} d D_{\max}, \\
&\| A^n \|_1 = \max_p \left(|A^n_{pp}| + \sum_{p\neq q}|A^n_{qp}|\right) \\
&\leq  2\, \max_p |A^n_{pp}| -1  + \gamma \Delta t \leq 1 +  4 \frac{\Delta t}{h^2} d D_{\max},
\end{align*}
and it follows that,
$
\|A^n\|_2 \leq \sqrt{\|A\|_\infty \|A\|_1} \leq 1 +  4 \frac{\Delta t}{h^2} d D_{\max}.
$

Using the results from~\cite{sigma_bound}  (similar to the Corollary \ref{cor: Ainv}), we obtain $\|(A^n)\inv\|_2 \leq \frac{1}{1-\gamma \Delta t } \leq  \big(1 + \gamma \Delta t\big)^2$ for $\gamma \Delta t \leq 0.5$.

\end{proof}

\begin{theorem} \label{thm: fd_complexity}
For a given error tolerance $\epsilon<1$, the query complexity of QLSFDA is,
\[
\mathcal{O}\bigg( \frac{e^{2\gamma T}}{\gamma } \frac{d^3}{q\epsilon}  \log \frac{1}{\epsilon}\bigg),
\]
where, $q = \|\rho(T)\|$, and it produces a normalized solution to the FPE (\ref{eq: fp}) within an $l_2$ error $\epsilon$ with a constant probability $\Omega(1)$ of success. The gate complexity is larger than the query complexity only by a logarithmic factor.
\end{theorem}
\begin{proof}
Similar to the proof of Theorem \ref{thm: complexity}, we bound the condition number $\kappa(L)=\|L\|_2 \|L\inv\|_2$ using following inequalities,
\begin{align*}
    \|L\|_2 &\leq \|I\|_2 + \max_{n=1,\dots, N_t} \|A^n\|_2 \\
& \leq 2 + 4 \frac{\Delta t}{h^2} d \, D_{\max},
\end{align*}
\begin{align*}
&\|L\inv\|_2 \leq \sum_{j=1}^{N_t} \max_{i=0, \dots, N_t-1} \| (A^i)^{-1}\|_2^j  \\
&\leq (1+ \gamma \Delta t) \frac{(1+ \gamma \Delta t)^{2N_t} -1}{\gamma \Delta t } \\
&\leq \frac{3}{2} \frac{1}{\gamma \Delta t } e^{2\gamma T}.
\end{align*}
Then,
\[
\kappa(L) \leq \frac{3}{2} \frac{e^{2 \gamma T}}{\gamma} \bigg(\frac{2}{\Delta t } + \frac{4d\, D_{\max}}{h^2} \bigg).
\]
Finally, using the choice of $h$, $\Delta t$ from \Cref{thm: fd_err} leads to the desired result.
\end{proof}

Note that QLSFDA has similar complexity as QLSCCA in terms of time $T$, dimension $d$ and accuracy $\epsilon$, and hence the Remark \ref{remark1} also applies to QLSFDA. However, we needed different assumptions to prove these complexity results, and hence the applicability of these algorithms may differ depending on the problem. 

\section{Conclusions}\label{sec:conc}
In this paper, we have developed an approach for solving the Fokker-Planck equation associated with nonlinear SDEs on quantum platforms. Specifically, we applied the Chang-Cooper and implicit finite difference schemes to discretize the FPE in space and time, and apply the QLSA to solve the resulting linear system. We perform detailed error and complexity analysis showing that our proposed quantum algorithms using either of these schemes, under certain conditions, can provably generate the solution to the FPE within a prescribed $\epsilon$ error with polynomial dependence on the state dimension $d$. In contrast, the SOA classical numerical approaches exhibit exponential dependence on $d$. In future work, we plan to explore Hamiltonian simulation framework in~\cite{jin2022quantum}  for solving the FPE and compare with the QLSA based framework developed in this paper.


\bibliographystyle{IEEETran}
\bibliography{references}

\begin{thebibliography}{10}
\providecommand{\url}[1]{#1}
\csname url@samestyle\endcsname
\providecommand{\newblock}{\relax}
\providecommand{\bibinfo}[2]{#2}
\providecommand{\BIBentrySTDinterwordspacing}{\spaceskip=0pt\relax}
\providecommand{\BIBentryALTinterwordstretchfactor}{4}
\providecommand{\BIBentryALTinterwordspacing}{\spaceskip=\fontdimen2\font plus
\BIBentryALTinterwordstretchfactor\fontdimen3\font minus
  \fontdimen4\font\relax}
\providecommand{\BIBforeignlanguage}[2]{{%
\expandafter\ifx\csname l@#1\endcsname\relax
\typeout{** WARNING: IEEEtran.bst: No hyphenation pattern has been}%
\typeout{** loaded for the language `#1'. Using the pattern for}%
\typeout{** the default language instead.}%
\else
\language=\csname l@#1\endcsname
\fi
#2}}
\providecommand{\BIBdecl}{\relax}
\BIBdecl

\bibitem{oksendal2003stochastic}
B.~{\O}ksendal, ``Stochastic differential equations,'' in \emph{Stochastic
  differential equations}.\hskip 1em plus 0.5em minus 0.4em\relax Springer,
  2003, pp. 65--84.

\bibitem{Rodkina2011}
\BIBentryALTinterwordspacing
A.~Rodkina and C.~Kelly, \emph{Stochastic Difference Equations and
  Applications}.\hskip 1em plus 0.5em minus 0.4em\relax Berlin, Heidelberg:
  Springer Berlin Heidelberg, 2011, pp. 1517--1520. [Online]. Available:
  \url{https://doi.org/10.1007/978-3-642-04898-2_568}
\BIBentrySTDinterwordspacing

\bibitem{feynman2018simulating}
R.~P. Feynman, ``Simulating physics with computers,'' in \emph{Feynman and
  computation}.\hskip 1em plus 0.5em minus 0.4em\relax CRC Press, 2018, pp.
  133--153.

\bibitem{berry2014high}
D.~W. Berry, ``High-order quantum algorithm for solving linear differential
  equations,'' \emph{Journal of Physics A: Mathematical and Theoretical},
  vol.~47, no.~10, p. 105301, 2014.

\bibitem{berry2017quantum}
D.~W. Berry, A.~M. Childs, A.~Ostrander, and G.~Wang, ``Quantum algorithm for
  linear differential equations with exponentially improved dependence on
  precision,'' \emph{Communications in Mathematical Physics}, vol. 356, no.~3,
  pp. 1057--1081, 2017.

\bibitem{childs2020quantum}
A.~M. Childs and J.-P. Liu, ``Quantum spectral methods for differential
  equations,'' \emph{Communications in Mathematical Physics}, vol. 375, no.~2,
  pp. 1427--1457, 2020.

\bibitem{childs2021high}
A.~M. Childs, J.-P. Liu, and A.~Ostrander, ``High-precision quantum algorithms
  for partial differential equations,'' \emph{Quantum}, vol.~5, p. 574, 2021.

\bibitem{costa2019quantum}
P.~C. Costa, S.~Jordan, and A.~Ostrander, ``Quantum algorithm for simulating
  the wave equation,'' \emph{Physical Review A}, vol.~99, no.~1, p. 012323,
  2019.

\bibitem{linden2020quantum}
N.~Linden, A.~Montanaro, and C.~Shao, ``Quantum vs. classical algorithms for
  solving the heat equation,'' \emph{arXiv preprint arXiv:2004.06516}, 2020.

\bibitem{montanaro2016quantum}
A.~Montanaro and S.~Pallister, ``Quantum algorithms and the finite element
  method,'' \emph{Physical Review A}, vol.~93, no.~3, p. 032324, 2016.

\bibitem{leyton2008quantum}
S.~K. Leyton and T.~J. Osborne, ``A quantum algorithm to solve nonlinear
  differential equations,'' \emph{arXiv preprint arXiv:0812.4423}, 2008.

\bibitem{liu2021efficient}
J.-P. Liu, H.~{\O}. Kolden, H.~K. Krovi, N.~F. Loureiro, K.~Trivisa, and A.~M.
  Childs, ``Efficient quantum algorithm for dissipative nonlinear differential
  equations,'' \emph{Proceedings of the National Academy of Sciences}, vol.
  118, no.~35, p. e2026805118, 2021.

\bibitem{surana2022carleman}
A.~Surana, A.~Gnanasekaran, and T.~Sahai, ``An efficient quantum algorithm for
  simulating polynomial differential equations,'' \emph{In Review. arXiv
  preprint arXiv:2212.10775}, 2022.

\bibitem{jin2022time}
S.~Jin, N.~Liu, and Y.~Yu, ``Time complexity analysis of quantum algorithms via
  linear representations for nonlinear ordinary and partial differential
  equations,'' \emph{arXiv preprint arXiv:2209.08478}, 2022.

\bibitem{joseph2020koopman}
I.~Joseph, ``Koopman--von {N}eumann approach to quantum simulation of nonlinear
  classical dynamics,'' \emph{Physical Review Research}, vol.~2, no.~4, p.
  043102, 2020.

\bibitem{lin2022koopman}
Y.~T. Lin, R.~B. Lowrie, D.~Aslangil, Y.~Suba{\c{s}}{\i}, and A.~T. Sornborger,
  ``Koopman--von {N}eumann mechanics and the {K}oopman representation: A
  perspective on solving nonlinear dynamical systems with quantum computers,''
  \emph{arXiv preprint arXiv:2202.02188}, 2022.

\bibitem{giannakis2022embedding}
D.~Giannakis, A.~Ourmazd, P.~Pfeffer, J.~Schumacher, and J.~Slawinska,
  ``Embedding classical dynamics in a quantum computer,'' \emph{Physical Review
  A}, vol. 105, no.~5, p. 052404, 2022.

\bibitem{evans2012introduction}
L.~C. Evans, \emph{An introduction to stochastic differential equations}.\hskip
  1em plus 0.5em minus 0.4em\relax American Mathematical Soc., 2012, vol.~82.

\bibitem{risken1996fokker}
H.~Risken, ``Fokker-{P}lanck equation,'' in \emph{The Fokker-Planck
  Equation}.\hskip 1em plus 0.5em minus 0.4em\relax Springer, 1996, pp. 63--95.

\bibitem{jin2022quantum}
S.~Jin, N.~Liu, and Y.~Yu, ``Quantum simulation of partial differential
  equations via {S}chrodingerisation: technical details,'' \emph{arXiv preprint
  arXiv:2212.14703}, 2022.

\bibitem{chen2017beating}
N.~Chen and A.~J. Majda, ``Beating the curse of dimension with accurate
  statistics for the {F}okker--{P}lanck equation in complex turbulent
  systems,'' \emph{Proceedings of the National Academy of Sciences}, vol. 114,
  no.~49, pp. 12\,864--12\,869, 2017.

\bibitem{sun2014numerical}
Y.~Sun and M.~Kumar, ``Numerical solution of high dimensional stationary
  {F}okker--{P}lanck equations via tensor decomposition and {C}hebyshev
  spectral differentiation,'' \emph{Computers \& Mathematics with
  Applications}, vol.~67, no.~10, pp. 1960--1977, 2014.

\bibitem{pmlr-v145-zhai22a}
\BIBentryALTinterwordspacing
J.~Zhai, M.~Dobson, and Y.~Li, ``A deep learning method for solving
  {F}okker-{P}lanck equations,'' in \emph{Proceedings of the 2nd Mathematical
  and Scientific Machine Learning Conference}, ser. Proceedings of Machine
  Learning Research, J.~Bruna, J.~Hesthaven, and L.~Zdeborova, Eds., vol.
  145.\hskip 1em plus 0.5em minus 0.4em\relax PMLR, 16--19 Aug 2022, pp.
  568--597. [Online]. Available:
  \url{https://proceedings.mlr.press/v145/zhai22a.html}
\BIBentrySTDinterwordspacing

\bibitem{sahai2013uncertainty}
T.~Sahai and J.~M. Pasini, ``Uncertainty quantification in hybrid dynamical
  systems,'' \emph{Journal of Computational Physics}, vol. 237, pp. 411--427,
  2013.

\bibitem{gardiner1985handbook}
C.~W. Gardiner \emph{et~al.}, \emph{Handbook of stochastic methods}.\hskip 1em
  plus 0.5em minus 0.4em\relax springer Berlin, 1985, vol.~3.

\bibitem{pichler2013numerical}
L.~Pichler, A.~Masud, and L.~A. Bergman, ``Numerical solution of the
  {F}okker--{P}lanck equation by finite difference and finite element
  methods—a comparative study,'' \emph{Computational Methods in Stochastic
  Dynamics: Volume 2}, pp. 69--85, 2013.

\bibitem{chang1970practical}
J.~Chang and G.~Cooper, ``A practical difference scheme for {F}okker-{P}lanck
  equations,'' \emph{Journal of Computational Physics}, vol.~6, no.~1, pp.
  1--16, 1970.

\bibitem{cc}
C.~Buet and S.~Dellacherie, ``On the {C}hang and {C}ooper scheme applied to a
  linear {F}okker-{P}lanck equation,'' \emph{Communications in Mathematical
  Sciences}, vol.~8, no.~4, pp. 1079--1090, 2010.

\bibitem{cc_analysis}
M.~Mohammadi and A.~Borz{\`\i}, ``Analysis of the {C}hang--{C}ooper
  discretization scheme for a class of {F}okker--{P}lanck equations,''
  \emph{Journal of Numerical Mathematics}, vol.~23, no.~3, pp. 271--288, 2015.

\bibitem{qlsa_optimal}
P.~C. Costa, D.~An, Y.~R. Sanders, Y.~Su, R.~Babbush, and D.~W. Berry,
  ``Optimal scaling quantum linear-systems solver via discrete adiabatic
  theorem,'' \emph{PRX Quantum}, vol.~3, no.~4, p. 040303, 2022.

\bibitem{nielsen2002quantum}
M.~A. Nielsen and I.~Chuang, ``Quantum computation and quantum information,''
  2002.

\bibitem{sigma_bound}
J.~M. Varah, ``A lower bound for the smallest singular value of a matrix,''
  \emph{Linear Algebra and its applications}, vol.~11, no.~1, pp. 3--5, 1975.

\bibitem{patankar2018numerical}
S.~Patankar, \emph{Numerical heat transfer and fluid flow}.\hskip 1em plus
  0.5em minus 0.4em\relax Taylor \& Francis, 2018.

\bibitem{implicit_fd_convg}
H.~Schroll, ``Convergence of implicit finite difference methods applied to
  nonlinear mixed systems,'' \emph{SIAM journal on numerical analysis},
  vol.~33, no.~3, pp. 997--1013, 1996.

\end{thebibliography}
\end{document}